\documentclass[11pt, reqno]{article}

\usepackage{amsmath}
\usepackage{amsthm}
\usepackage{amscd}
\usepackage{amsfonts}

\usepackage[mathscr]{eucal}

\theoremstyle{plain}
\newtheorem{theorem}{Theorem}[section]

\newtheorem{lemma}[theorem]{Lemma}

\theoremstyle{definition}
\newtheorem{definition}[theorem]{Definition}

\newcommand{\intav}[1]{\mathchoice {\mathop{\vrule width 6pt height 3 pt depth  -2.5pt
\kern -8pt \intop}\nolimits_{\kern -6pt#1}} {\mathop{\vrule width
5pt height 3  pt depth -2.6pt \kern -6pt \intop}\nolimits_{#1}}
{\mathop{\vrule width 5pt height 3 pt depth -2.6pt \kern -6pt
\intop}\nolimits_{#1}} {\mathop{\vrule width 5pt height 3 pt depth
-2.6pt \kern -6pt \intop}\nolimits_{#1}}}

\newcommand{\N}{\mathbb{N}}

\newcommand{\Rn}{\mathbb{R}^n}

\newlength{\hchng}
\newlength{\vchng}
\title{A geometric tangential approach to sharp regularity for degenerate evolution equations}

\author{Eduardo V. Teixeira \quad $\&$ \quad Jos\'e Miguel Urbano}

\date{}
\begin{document}

\maketitle

\begin{abstract}

That the weak solutions of degenerate parabolic pdes modelled on the inhomogeneous $p-$Laplace equation
$$u_t - \mathrm{div} \left( |\nabla u|^{p-2} \nabla u \right) = f \in L^{q,r} , \quad p>2$$ 
are $C^{0,\alpha}$, for some $\alpha \in (0,1)$, is known for almost 30 years. What was hitherto missing from the literature was a precise and sharp knowledge of the H\"older exponent $\alpha$ in terms of $p, q, r$ and the space dimension $n$. We show in this paper that 
$$\alpha = \frac{(pq-n)r-pq}{q[(p-1)r-(p-2)]} ,$$
using a method based on the notion of geometric tangential equations and the intrinsic scaling of the $p-$parabolic operator. The proofs are flexible enough to be of use in a number of other nonlinear evolution problems. 

\bigskip

\noindent {\bf Keywords:} {Degenerate parabolic equations; sharp H\"older regularity; tangential equations; intrinsic scaling.  \smallskip }

\end{abstract}

\section{Introduction}

The understanding of the local behaviour of solutions to singular and degenerate parabolic equations has witnessed an impressive progress in the last three decades. At the heart of most developments lies a single unifying idea, namely that regularity results have to be interpreted in an intrinsic geometric configuration, a sort of signature to each particular \textit{pde}. The pioneering work of DiBenedetto \cite{DiB} was the starting point to a theory that has, in many aspects, reached its maturity (cf. \cite{DGV} and \cite{U} for recent accounts). 

A central aspect in this endeavour has always been the H\"older continuity of bounded weak solutions, which ultimately follows from Harnack type inequalities. Although powerful, this approach only provides qualitative estimates that depend solely on the structure of the equations and thus hold in a very general setting. The quest for precise, quantitative derivations of the H\"older exponent has hitherto eluded the community, the only exception being the two-dimensional result in \cite{IM} concerning $p-$harmonic functions. This type of quantitative information, apart from its own intrinsic value, plays an important role in the analysis of a number of qualitative issues for parabolic pdes, such as blow-up analysis, Liouville type results, free boundary problems, and so forth. 

The main goal of this paper is to fill this gap, bringing the theory to a new level of understanding. We show that weak solutions of degenerate $p-$parabolic equations whose prototype is 
\begin{equation} \label{eq_intro01}
	u_t - \mathrm{div} \left( |\nabla u|^{p-2} \nabla u \right) = f \in L^{q,r} , \quad p\ge 2
\end{equation}
are locally of class $C^{0,\alpha}$ in space, with
$$
	\alpha:= \frac{(pq-n)r-pq}{q[(p-1)r-(p-2)]},
$$
a precise and sharp expression for the H\"older exponent in terms of $p$, the integrability of the source and the space dimension $n$. We also show that $u$ is of class $C^{0,\frac{\alpha}{\theta}}$ in time, where $\theta$ is the $\alpha-$interpolation between $2$ and $p$. What makes the parabolic case more delicate to analyse is the inhomogeneity in the equation, the fact that it scales differently with respect to space and time. It is worth stressing that the integrability in time (respectively, in space) of the source affects the regularity in space (respectively, in time) of the solution. 

To highlight the extent to which our result is sharp, we project it into the state of the art of the theory. For the linear case $p=2$, we obtain
$$
	\alpha = 1 - \left( \frac{2}{r} + \frac{n}{q} -1 \right),
$$
which is the optimal H\"older exponent for the non-homogeneous heat equation, and is in accordance with estimates obtained by energy considerations.
When $p\rightarrow \infty$, we have $\alpha \rightarrow 1^-$, which gives an indication of the expected locally Lipschitz regularity for the case of the parabolic infinity-Laplacian. When the source $f$ is independent of time, or else bounded in time, that is $r=\infty$, we obtain
$$
	\alpha = \frac{pq-n}{q(p-1)} = \frac{p}{p-1} \cdot \frac{q-\frac{n}{p}}{q},
$$
which is exactly the optimal exponent obtained in \cite{T1} for the elliptic case. It might also be interesting to compare our optimal result with the estimates from \cite[Section 4]{M}, and also with the continuity estimates on $p-$parabolic obstacle problems from \cite{KMN}.

Within the general theory of $p-$parabolic equations, our result reveals a surprising feature. From the applied point of view, it is relevant to know what is the effect on the diffusion properties of the model as we dim the exponent $p$. Na\"ive physical interpretations could indicate that the higher the value of $p$, the less efficient should the diffusion properties of the $p-$parabolic operator turn out to be, \textit{i.e.}, one should expect a less efficient smoothness effect of the operator. For instance, this is verified in the sharp regularity estimate for $p-$harmonic functions in 2D \cite{IM}. On the contrary, our estimate implies that for $p-$parabolic inhomogeneous equations, the H\"older regularity theory improves as $p$ increases. In fact, a direct computation shows
$$
	\text{sign}\left ( \partial_p \alpha(p,n,q,r) \right) = \text{sign} \left ( q(2-r) + nr \right ) =+1,
$$
in view of standard assumptions on the integrability exponents of the source term.

Although regularity estimates for degenerate evolution equations have been successfully obtained in great generality (cf. \cite{KL}, \cite{AM}), explicit expressions for the H\"older exponent of continuity for weak solutions have only been known in the linear setting. For nonlinear equations, the classical tools from harmonic analysis, such as singular integrals, are precluded from being used and an entirely new approach is needed. The new estimates we obtain are striking in their simplicity but perhaps the most relevant contribution we offer is the technique employed. We develop a method based on the notion of geometric tangential equations, which explores the intrinsic scaling of the $p-$parabolic operator and the integrability of the forcing term. By means of  appropriate scaled iterative arguments, we show that at each inhomogeneous equation there is a universal tangential space formed by $C^{0,1}$ in space and $C^{0,\frac{1}{2}}$ in time functions. The method then imports such regularity back to the original equation, properly corrected through the scaling used to access the tangential space. The method is new to the field and robust enough to be adapted to other evolutionary problems, as well as to a number of other issues in the theory.

\section{Preliminary tools} 

Let $U \subset \Rn$ be open and bounded, and $T>0$. We consider the space-time domain $U_T=U\times (0,T)$. We work with the prototype inhomogeneous equation
\begin{equation} \label{p-Lap}
u_t - \mathrm{div} \left( |\nabla u|^{p-2} \nabla u \right) = f  \quad \textrm{in} \ \ U_T,
\end{equation}
with a source term $f \in L^{q,r} (U_T) \equiv L^r(0,T;L^q(U))$ satisfying 
\begin{equation} \label{borderline1}
\frac{1}{r}+\frac{n}{pq} <1
\end{equation}
and
\begin{equation}\label{borderline2}
\frac{2}{r}+\frac{n}{q} >1
\end{equation}

The first assumption is the standard minimal integrability condition that guarantees the existence of bounded weak solutions, while \eqref{borderline2} defines the borderline setting for optimal H\"older type estimates. For instance, when $r=\infty$,  conditions  \eqref{borderline1} and  \eqref{borderline2} enforce 
$$
	\dfrac{n}{p} < q < n,
$$
which corresponds to the known range of integrability required in the elliptic theory for local $C^{0,\alpha}$ estimates to be available. 

\medskip

We start with the definition of weak solution to \eqref{p-Lap}.

\begin{definition} \label{def ws} We say a function 
$$u\in C_\text{loc} \left( 0,T; L^2_\text{loc}(U) \right ) \cap L^p_\text{loc}\left(0,T; W^{1,p}_\text{loc}(U) \right )$$ 
is a weak solution to \eqref{p-Lap} if, for every compact $K \subset U$ and every subinterval $[t_1,t_2] \subset (0,T]$, there holds
$$
	\int_K u\varphi \, dx \Big|_{t_1}^{t_2} + \int_{t_1}^{t_2} \int_K \left\{ -u \varphi_t + |\nabla u|^{p-2} \nabla u \cdot \nabla \varphi \right\} \, dx dt = \int_{t_1}^{t_2} \int_K f \varphi \, dx dt,
$$
for all $\varphi \in H^1_\text{loc} \left(0,T; L^2(K) \right) \cap L^p_\text{loc} \left(0,T; W^{1,p}_0(K) \right)$.
\end{definition}
The following alternative definition makes use of the Steklov average of a function $v \in L^1(U_T)$, defined for $0<h<T$ by
$$v_h := \left\{
\begin{array}{ccl}
\frac{1}{h} \displaystyle \int_t^{t+h} v(\cdot,\tau) d\tau & {\rm if} & t \in (0,T-h]\\
& & \\
0 & {\rm if} & t \in (T-h, T],
\end{array} \right. $$
and circumvents the difficulties related to the low regularity in time. In fact, these difficulties are more of a technical nature since the time derivative $u_t$ is shown in \cite{Lindq} to be an element of a certain Lebesgue space. 

\begin{definition} \label{def ws steklov} We say a function 
$$u\in C_\text{loc} \left( 0,T; L^2_\text{loc}(U) \right ) \cap L^p_\text{loc}\left(0,T; W^{1,p}_\text{loc}(U) \right )$$ 
is a weak solution to \eqref{p-Lap} if, for every compact $K \subset U$ and every $0<t<T-h$, there holds
\begin{equation} \label{def_steklov}
	\int_{K \times\{t\}} \left\{ (u_h)_t \varphi + \left( |\nabla u|^{p-2} \nabla u \right)_h  \cdot \nabla \varphi \right\} \, dx = \int_{K \times\{t\}} f_h \varphi \, dx,
\end{equation}
for all $\varphi \in W^{1,p}_0(K)$.
\end{definition}

One key ingredient in our analysis is the following Caccioppoli-type energy estimate enjoyed by weak solutions of equation \eqref{p-Lap}.

\begin{lemma}[Caccioppoli estimate] \label{energy est} 
Let $u$ be a weak solution to \eqref{p-Lap}. Given $K  \times [t_1,t_2] \subset U \times (0,T]$, there exists a constant $C$, depending only on $n$, $p$, $K  \times [t_1,t_2]$ and $\|f\|_{L^{q,r}}$, such that
\begin{equation} \label{caccioppoli}
\sup_{t_1<t<t_2} \int_K u^2 \xi^p \, dx + \int_{t_1}^{t_2} \int_K \left| \nabla u \right|^p \xi^p \, dx dt \leq  \int_{t_1}^{t_2} \int_K |u|^p \left( \xi^p +\left| \nabla \xi \right|^p \right) \, dx dt  
\end{equation}
$$+ \int_{t_1}^{t_2} \int_K u^2 \xi^{p-1} \left| \xi_t \right| \, dx dt+ \| f \|_{q,r},
$$
for every $\xi \in {\mathcal C}_0^\infty \left( K \times (t_1, t_2)\right)$ such that $\xi \in [0,1]$.
\end{lemma}

\begin{proof}
Choose $\varphi = u_h \xi^p$ as a test function in \eqref{def_steklov} and perform the usual combination of integrating in time, passing to the limit in $h \rightarrow 0$ and applying Young's inequality to derive the estimate. 
\end{proof}

We finally recall that, for a function belonging to $L^p (Q)$, its averaged norm is
$$\| v\|_{p, avg,Q} := \left( \, \intav{Q} |v|^p \, dx dt \right)^{1/p} = |Q|^{-1/p} \| v\|_{p,Q},$$
where, as usual, the integral average is defined by
$$
	\intav{A} \psi  = \dfrac{1}{|A|} \int_{A} \psi.
$$

\section{Sharp H\"older estimate}

We start by fixing universal constants, that depend only on the data. The intrinsic exponent to equation \eqref{p-Lap}, with $f\in L^{q,r}$, is
\begin{equation}\label{alpha}
	\alpha := \frac{(pq-n)r-pq}{q[(p-1)r-(p-2)]} = \frac{p\left( 1-\displaystyle \frac{1}{r}-\frac{n}{pq}\right)}{\left( \displaystyle\frac{2}{r} + \frac{n}{q}-1\right) + p\left( 1-\displaystyle\frac{1}{r}-\frac{n}{pq}\right)},
\end{equation}
which, in view of \eqref{borderline1} and \eqref{borderline2}, satisfies $0<\alpha<1$.  Next, let 
\begin{equation} \label{theta}
\theta := \alpha + p-(p-1)\alpha =  p - (p-2)\alpha = \alpha 2 + (1-\alpha) p.
\end{equation}
Clearly $2<\theta<p$, since $0 < \alpha < 1$. For such $\theta$, we define the intrinsic $\theta-$parabolic cylinder
$$G_\tau := \left( - \tau^{\theta}, 0 \right) \times B_\tau(0), \quad \tau >0.$$

We first establish a key compactness result that states that if the source term $f$ has a small norm in $L^{q,r}$, then a solution $u$ to \eqref{p-Lap} is close to a $p-$caloric function in an inner subdomain.

\begin{lemma}[Approximation to $p-$caloric functions] \label{compactness}
For every $\delta >0$, there exists $0< \epsilon \ll 1$, such that if $\|f\|_{L^{q,r}(G_1)} \leq \epsilon$ and $u$ is a local weak solution of \eqref{p-Lap} in $G_1$, with $\|u\|_{p,avg,G_1} \leq 1$, then there exists a $p-$caloric function $\phi$ in $G_{1/2}$, i.e.,  
\begin{equation} \label{p-Lap homogeneous}
\phi_t - \mathrm{div} \left( |\nabla \phi|^{p-2} \nabla \phi \right) = 0 \quad \textrm{in} \ \ G_{1/2},
\end{equation}
such that 
\begin{equation} \label{close}
	\| u-\phi \|_{p,avg,G_{1/2}}  \leq \delta.
\end{equation}
\end{lemma}

\begin{proof} Suppose, for the sake of contradiction, that the thesis of the lemma fails. That is, assume, for some $\delta_0>0$, that there exists a sequence  
$$(u^j)_j \in C_\text{loc} \left( -1,0; L^2_\text{loc}(B_1) \right) \cap L^p_\text{loc}\left(-1,0; W^{1,p}_\text{loc}(B_1) \right)$$
and a sequence $(f^j)_j \in L^{q,r}(G_1)$, such that
\begin{eqnarray}
	\label{comp eq01}  u^j_t - \mathrm{div} \left( |\nabla u^j |^{p-2} \nabla u^j \right) &=& f^j \ \text{ in }  G_1; \\
	\label{comp eq02} \|u^j\|_{p,avg,G_1} &\leq& 1; \\
	\label{comp eq03}  \|f^j\|_{L^{q,r}(G_1)} &\leq& \frac{1}{j}
\end{eqnarray}
but still, for any $j$ and any $p-$caloric function $\phi$ in $G_{1/2}$, 
\begin{equation} \label{comp eq04} 
\| u^j-\phi \|_{p,avg,G_{1/2}}  > \delta_0.
\end{equation} 

Fix a cutoff function $\xi \in C_0^\infty (G_1)$, such that $\xi \in [0,1]$, $\xi \equiv 1$ in $G_{1/2}$ and $\xi \equiv 0$ near $\partial_p G_1$. From the Caccioppoli estimate, using the notation
$$V(I\times U) = L^\infty \left(I; L^2 (U) \right) \cap L^p \left(I; W^{1,p}(U) \right),$$
we obtain
\begin{eqnarray*}
\| u^j \|_{V(G_{1/2})} & \leq & \sup_{-1<t<0} \int_{B_1} (u^j)^2 \xi^p \, dx + \int_{-1}^{0} \int_{B_1} \left| \nabla u^j \right|^p \xi^p \, dx dt \\
& \leq & \int_{-1}^{0} \int_{B_1} \left\{ |u^j|^p \left( \xi^p +\left| \nabla \xi \right|^p \right)  + (u^j)^2 \xi^{p-1} \left| \xi_t \right| \right\} dx dt+ \|f^j\|_{L^{q,r}(G_1)}\\
& \leq & c\, \|u^j\|^p_{p,avg,G_1} + c^\prime \|u^j\|^2_{2,avg,G_1} + 1/j\\
& \leq & c.
\end{eqnarray*}

A control of the time derivative, along the lines of \cite{Lindq} (see also \cite{AMS}), gives
$$\|u^j_t\|_{L^{s,1}(G_{1/2})} \leq c,$$
with $s=\min \left\{ q, \frac{p}{p-1}\right\} < p$.
We now use a classical compactness result (cf. \cite[Corollary 4]{Sim}), with
$$W^{1,p}  \hookrightarrow L^p \subset L^s,$$
to conclude that
$$u^j \longrightarrow \psi,$$
strongly in $L^p (G_{1/2})$, in addition to the weak convergence in $V(G_{1/2})$.

Passing to the limit in \eqref{comp eq01}, we find that
$$\psi_t - \mathrm{div} \left( |\nabla \psi|^{p-2} \nabla \psi \right) = 0 \quad \textrm{in} \ \ G_{1/2},$$
which contradicts \eqref{comp eq04}, for $j \gg 1$. The proof is complete.
\end{proof}

Next, by means of geometric iteration, we shall establish the optimal H\"older continuity for solutions to the heterogeneous $p-$parabolic equation \eqref{p-Lap}. Our approach explores the approximation by $p-$caloric functions, given by Lemma \ref{compactness}, and  the fact that $p-$caloric functions are \textit{universally} Lipschitz continuous in space and $C^{0, \frac{1}{2}}$ in time.  The following is the crucial first iterative step. 

\begin{lemma} \label{constant}
Let $0<\alpha<1$ be fixed. There exists $\epsilon >0$ and $0< \lambda \ll 1/2$, depending only on $p$, $n$ and $\alpha$, such that if $\|f\|_{L^{q,r}(G_{1})} \leq \epsilon$ and $u$ is a local weak solution of \eqref{p-Lap} in $G_{1}$, with $\|u\|_{p,avg,G_{1}}  \le 1$, then there exists a universally bounded constant $c_0$ such that 
\begin{equation} \label{supi}
\| u -c_0 \|_{p,avg,G_{\lambda}} \leq \lambda^\alpha.
\end{equation}
\end{lemma}

\proof

Take $0<\delta <1$, to be chosen later, and apply Lemma \ref{compactness} to obtain $0< \epsilon \ll 1$ and a $p-$caloric function $\phi$ in $G_{1/2}$, such that 
$$\| u-\phi \|_{p,avg,G_{1/2}}  \leq \delta.$$
Observe that 
\begin{equation} \label{cirne}
\| \phi \|_{p,avg,G_{1/2}} \leq \| u-\phi \|_{p,avg,G_{1/2}} + \| u \|_{p,avg,G_{1}} \leq \delta +1 \leq 2.
\end{equation}
Since $\phi$ is $p-$caloric, it follows from standard theory that $\phi$ is universally $C_\text{loc}^{0,\frac{1}{2}}$ in time and $C_\text{loc}^{0,1}$ in space.  That is, for $\lambda \ll 1$, to be chosen soon, we have    
$$
	\sup_{(x,t) \in G_\lambda} | \phi(x,t)-\phi (0,0) | \leq  C \, \lambda,
$$
for $C>1$ universal. In fact, for $(x,t) \in G_\lambda$,
\begin{eqnarray*}
 | \phi(x,t)-\phi (0,0) | & \leq &  | \phi(x,t)-\phi (0,t) | +  | \phi(0,t)-\phi (0,0) |\\
 & \leq & C^\prime\, |x-0| + C^{\prime \prime} \, |t-0|^{\frac{1}{2}} \\
 & \leq & C^\prime \, \lambda + C^{\prime \prime} \, \lambda^{ \frac{\theta}{2}} \leq C\, \lambda,
\end{eqnarray*} 
since $\theta >2$. We can therefore estimate
\begin{eqnarray}
	\| u(x,t) - \phi (0,0) \|_{p,avg,G_{\lambda}} & \leq &  \|u(x,t)-\phi(x,t)\|_{p,avg,G_{\lambda}}  \nonumber \\
	&+& \|\phi(x,t) - \phi(0,0)\|_{p,avg,G_{\lambda}} \nonumber \\
& \leq & \delta + C\, \lambda. \label{beach park}
\end{eqnarray}
Note that we will choose $\lambda \ll 1/2$ and thus
$$G_{\lambda} =  (-\lambda^{\theta}, 0) \times  B_\lambda \subset  (-(1/2)^{\theta}, 0) \times  B_{1/2} = G_{1/2}.$$

We put $c_0 := \phi(0,0)$, observing that, due to \eqref{cirne} and the fact that $\phi$ is $p-$caloric, $c_0$ is universally bounded.
The next step is to fix the constants. We choose $\lambda \ll 1/2$ so small that
$$C\, \lambda \leq \frac{1}{2} \lambda^\alpha$$
and then we define
$$\delta = \frac{1}{2} \lambda^\alpha$$
thus fixing, via Lemma \ref{compactness}, also $\epsilon >0$.
The lemma now follows from estimate \eqref{beach park} with the indicated choices.
\qed

\bigskip

Our next step accounts to iterating Lemma \ref{constant} in the appropriate geometric scaling. 

\begin{theorem} \label{iteration}
Under the conditions of the previous lemma, there exists a convergent sequence of real numbers $\{c_k\}_{k\ge 1}$, with
\begin{equation} \label{const}
|c_k - c_{k+1}| \leq c(n,p) \left( \lambda^\alpha \right)^k,
\end{equation}
such that
\begin{equation} \label{conclusion}
\|u-c_k\|_{p,avg,G_{\lambda^{k}}} \leq \left(\lambda^k\right)^\alpha.
\end{equation}
\end{theorem}

\proof

The proof is by induction on $k \in \N$. For $k=1$, \eqref{conclusion} holds due to Lemma \ref{constant}, with $c_1=c_0$. 
Suppose the conclusion holds for $k$ and let's show it also holds for $k+1$. We start by defining the function $v \colon G_{1} \to \mathbb{R}$ by
\begin{equation}\label{def v}
	v(x,t) = \frac{u(\lambda^k x, \lambda^{k \theta} t) - c_k}{\lambda^{\alpha k}}.
\end{equation}
We compute 
$$v_t(x,t)=\lambda^{k \theta - \alpha k} u_t(\lambda^k x, \lambda^{k \theta} t)$$
and
$$\hspace*{-5cm} \mathrm{div} \left( |\nabla v(x,t)|^{p-2} \nabla v(x,t) \right) = $$
\vspace*{-0.2cm}
$$\hspace*{3cm} = \lambda^{pk-(p-1) \alpha k} \mathrm{div} \left( |\nabla u(\lambda^k x, \lambda^{k \theta} t)|^{p-2} \nabla u(\lambda^k x, \lambda^{k \theta} t) \right)$$
to conclude, recalling \eqref{theta}, that
$$v_t - \mathrm{div} \left( |\nabla v|^{p-2} \nabla v \right) = \lambda^{pk-(p-1) \alpha k} f(\lambda^k x, \lambda^{k \theta} t) = \tilde{f} (x,t).$$
We now compute
\begin{eqnarray}
\| \tilde{f} \|_{L^{q,r} (G_1)}^r &=& \int_{-1}^0 \left( \int_{B_1}  |\tilde{f} (x,t)|^{q} dx \right)^{r/q}  dt  \label{bound of f} \\
& = & \int_{-1}^0  \left(  \int_{B_1}\lambda^{[pk-(p-1) \alpha k]q} |f(\lambda^k x, \lambda^{k \theta} t)|^{q} dx\right)^{r/q} dt \nonumber \\
& = & \int_{-1}^0  \left(  \int_{B_{\lambda^k}}\lambda^{[pk-(p-1) \alpha k]q-kn} |f( x, \lambda^{k \theta} t)|^{q} dx\right)^{r/q} dt \nonumber \\
& = & \lambda^{\{[pk-(p-1) \alpha k]q-kn\}\frac{r}{q}} \int_{-1}^0  \left(  \int_{B_{\lambda^k}} |f( x, \lambda^{k \theta} t)|^{q} dx\right)^{r/q} dt \nonumber \\
& = & \lambda^{\{[pk-(p-1) \alpha k]q-kn\}\frac{r}{q}-k \theta} \int_{-\lambda^{k \theta}}^0  \left(  \int_{B_{\lambda^k}} |f( x,  t)|^{q} dx\right)^{r/q} dt. \nonumber
\end{eqnarray}
Due to the crucial and sharp choice \eqref{alpha} of $\alpha$, we have, recalling again \eqref{theta},
$$\{[pk-(p-1) \alpha k]q-kn\}\frac{r}{q}-k \theta= 0.$$
We go back to \eqref{bound of f} to conclude
$$\| \tilde{f} \|_{L^{q,r} (G_1)} \leq \| f \|_{L^{q,r} ((-\lambda^{k \theta},0)\times B_{\lambda^k})} \leq \| f \|_{L^{q,r} (G_1)}\leq \epsilon,$$
which entitles $v$ to Lemma  \ref{constant} (note that $\|v\|_{p,avg,G_{1}}  \le 1$, due to the induction hypothesis). 

It then follows that there exists a constant $\tilde{c_0}$, with $|\tilde{c_0}| \leq c(n,p)$, such that  
$$\| v -\tilde{c_0}  \|_{p,avg,G_{\lambda} } \leq \lambda^\alpha,$$
which is the same as 
$$\| u -c_{k+1}  \|_{p,avg,G_{\lambda^{k+1}}} \leq \lambda^{\alpha (k+1)},$$
for $c_{k+1} := c_k +\tilde{c_0} \lambda^{\alpha k} $; the induction is complete. We readily observe that 
$$\left| c_{k+1} - c_k \right| \leq  c(n,p) \left(\lambda^{\alpha}\right)^k, $$
thus obtaining also \eqref{const}.
\qed

\bigskip

\begin{theorem} \label{holder}
A locally bounded weak solution of \eqref{p-Lap}, with $f \in L^{q,r}$, satisfying \eqref{borderline1}--\eqref{borderline2}, is locally H\"older continuous in the space variables, with exponent 
$$\alpha = \frac{(pq-n)r-pq}{q[(p-1)r-(p-2)]}$$ 
and locally H\"older continuous in time with exponent $\frac{\alpha}{\theta}$. In addition, there exists a constant $C$, that depends only on $p, n$, $\|f\|_{q,r}$ and $\|u\|_{p, \text{avg}, G_1}$, such that
$$
	\|u\|_{C^{0;\alpha,\alpha/\theta}(G_{1/2})} \le C . 
$$
\end{theorem}

\proof
We start by observing (see also \cite[section 7]{ART}) that the smallness regime required in the assumptions of Theorem \ref{iteration} is not restrictive since we can fall into that framework by scaling and contraction. Indeed, given a solution $u$, let 
$$v(x,t) = \varrho u ( \varrho^a x, \varrho^{(p-2) + ap} t)$$
($\varrho$, $a$ to be fixed), which is a solution of \eqref{p-Lap} with 
$$ \tilde{f} (x,t) =  \varrho^{(p-1) + ap} f ( \varrho^a x, \varrho^{(p-2) + ap} t).$$
We choose $a>0$ such that 
$$a < \frac{2}{n+p} \qquad \textrm{and} \qquad [(p-1)+ap] r - a (n+p) - (p-2) > 0,$$
which is always possible (observe that the second condition holds for $a=0$ and use its continuity with respect to $a$), and then $0<\varrho <1$ such that
$$\| v \|_{p,avg,G_{1}}^p \leq \varrho^{2-a(n+p)} \|u\|_{p,avg,G_1}^p \leq 1$$
and
$$\| \tilde{f} \|_{L^{q,r} (G_1)}^r = \varrho^{[(p-1)+ap] r - a (n+p) - (p-2)} \| f \|_{L^{q,r} (G_1)}^r \leq \epsilon^r.$$

Due to \eqref{const}, the sequence $\{c_k\}_{k\ge 1}$ is convergent and we put
$$\bar{c} := \lim_{k \to \infty} c_k.$$
It follows from \eqref{conclusion} that, for arbitrary $0<r<\frac{1}{2}$,
$$ \intav{G_r} |u -\bar{c}  |^p \, dx dt  \leq C r^{p\alpha}.$$
Standard covering arguments, a remark in \cite[Lemma 3.2]{T1} and the characterisation of H\"older continuity of Campanato-Da Prato give the local $C^{0;\alpha, \alpha/\theta}$ -- continuity and thus the result.
\qed

\section{Generalizations and beyond} \label{Sct final}

The ideas and methods employed in this paper only explore the degenerate $p-$structure of the operator. The underlying heuristics is to interpret the homogeneous problem as the geometric tangential equation of its inhomogeneous counterpart, for small perturbations $f \in L^{r,q}$, $\|f\|_{r,q} \ll 1$. The proofs adapt to more general degenerate parabolic equations 
\begin{equation} \label{GeneralEq}
	u_t - \text{div}\, \mathscr{A}(x,t,Du)  = f \in L^{r,q}
\end{equation}
satisfying the usual structure assumptions for $p\geq 2$.

We briefly comment on the modifications required. Lemma \ref{energy est} is based on pure energy considerations, thus the very same proof works in the general case. Lemma \ref{compactness} can be carried out universally in the structural class of operators, provided integrability bounds for the time-derivative are available (cf. \cite[section 7]{AMS}, where a more general version of the result in \cite{Lindq} on this issue is proved).
As for Lemma  \ref{constant}, the very same proof works since solutions to the general homogeneous equation are also Lipschitz in space and $C^{0;1/2}$ in time. The only modification occurs when we iterate Lemma \ref{constant}. The rescaled function $v$ defined in \eqref{def v} now solves the equation
$$
	v_t - \mathrm{div}\, \mathscr{A}_k(x,t,Dv) = \lambda^{pk-(p-1) \alpha k} f(\lambda^k x, \lambda^{k \theta} t),
$$
where
$$
	\mathscr{A}_k(x,t,\xi) := \left (\lambda^{-\alpha k}\right )^{1-p} \mathscr{A}(\lambda^k x, \lambda^{\theta k}t, \lambda^{-\alpha k}\xi)
$$
belongs to the same structural class of $\mathscr{A}$. In particular, $v$ is entitled to the conditions of Lemma \ref{constant} and the proof then follows exactly as before.

\bigskip
We would like to conclude by explaining how the  idea of finding geometrical tangential equations can be employed to derive analytical tools for $p-$parabolic operators, continuously on $p$. For instance, one can access regularity estimates for degenerate parabolic equations by interpreting the heat operator as the tangential equation obtained when we differentiate the family of $p-$parabolic operators with respect to the exponent $p$, near $p=2$. 

It is possible to obtain a universal compactness device. 
Let $Q_\tau := I_\tau \times B_\tau = \left( -\tau, \tau \right) \times B_{\tau}$. We fix $M_0 \gg 2$ and work within the range $p \in [2, M_0]$. 

\begin{lemma}[Uniform in $p$ compactness] \label{unif compac}
Given $\delta >0$, there exists $\epsilon>0$, depending only on  $n$, $M_0$ and $\delta$, such that if $q\in [2, M_0]$, $u$ is a $q-$caloric function in $Q_1$, with $|u| \leq 1$, and $|q-p| < \epsilon$, then we can find a $p-$caloric function $w$ in $Q_{1/2}$, with $|w|\leq 1$, such that
\begin{equation} \label{closy}
	\sup_{Q_{\frac{1}{2}}} |w-u|  \leq \delta.
\end{equation}
\end{lemma}
\begin{proof} 
Suppose, for the sake of contradiction, that the thesis of the lemma does not hold true. This means that for a certain $\delta_0 >0$, there exist sequences $(q_j)_j$, $(u_j)_j$ and $(p_j)_j$, with
\begin{equation} \label{lad}
\left \{
\begin{array}{l}
q_j \in [2,M_0];\\
\\
(u_j)_t - \mathrm{div} \left( |\nabla u_j|^{q_j-2} \nabla u_j \right) = 0  \quad \textrm{in} \ \ Q_1;\\
\\
|u_j|\leq 1;\\
\\
|p_j-q_j| \leq \frac{1}{j}
\end{array}
\right.
\end{equation}
but such that, for every $p_j-$caloric function $w$ in $Q_{\frac{1}{2}}$,
\begin{equation} \label{hanninha}
\sup_{Q_{\frac{1}{2}}} |u_j-w|  > \delta_0.
\end{equation}
By compactness, we have, up to subsequences,
\begin{equation} \label{gangnam}
q_j \rightarrow q_\infty \in [2, M_0]
\end{equation}
and, from the last assertion in \eqref{lad}, also $p_j \rightarrow q_\infty $.
As in the proof of Lemma  \ref{compactness},  up to a subsequence, $u_j \rightarrow u_\infty$ in the appropriate space. Since $q_j  \rightarrow q_\infty$, by stability (cf. \cite{KP}), we can pass to the limit in the equation satisfied by the $u_j$ to conclude that $u_\infty$ is $q_\infty-$caloric in $Q_{\frac{2}{3}}$.

We now solve, for each $p_j$, the following boundary value problem
\begin{equation} \label{bvp}
\left \{
\begin{array}{rcl}
(w_j)_t - \mathrm{div} \left( |\nabla w_j|^{p_j-2} \nabla w_j \right) = 0  & \textrm{in}  & Q_{\frac{2}{3}}\\
\\
w_j = u_\infty & \textrm{on}  & \partial Q_{\frac{2}{3}}\\
\end{array}
\right.
\end{equation}
and pass to the limit in $j$, concluding that also $w_j \rightarrow u_\infty$ uniformly in $ Q_{\frac{1}{2}}$.

Finally, choosing $j$ sufficiently large, we obtain
$$|u_j - w_j| \leq |u_j - u_\infty| + |w_j - u_\infty| \leq \frac{\delta_0}{2} + \frac{\delta_0}{2} = \delta_0 \quad \textrm{in} \ \ Q_{\frac{1}{2}}$$
which is a contradiction to \eqref{hanninha}.
\end{proof}

Heuristically, Lemma \ref{unif compac} implies the continuity of the underlying regularity theory for $p-$parabolic operators with respect to $p$. In particular, improved sharp H\"older estimates can be derived by these methods for problems governed by $p-$parabolic operators, near the heat equation, \textit{i.e.}, for $p$ close to 2. We leave the development of these heuristics for a future work.

\bigskip

\noindent \textbf{Acknowledgements.} This work was developed in the framework of the Brazilian Program "Ci\^encia sem Fronteiras". The authors thank Giuseppe Mingione and Rico Zacher for interesting conversations on the topic of the paper, and acknowledge the warm hospitality of Universidade Federal do Cear\'a and Universidade de Coimbra. 

Research of JMU partially supported by projects UTAustin/MAT/0035/2008, PTDC/MAT/098060/2008 and  UTA-CMU/MAT/0007/2009 and by the Centro de Matem\'ati\-ca da Universidade de Coimbra (CMUC), funded by the European Regional Development Fund through the program COMPETE and by the Funda\c c\~ao para a Ci\^encia e a Tecnologia (FCT) under the project PEst-C/MAT/UI0324/2011. Research of EVT partially supported by CNPq--Brazil.

\bigskip

\vspace{1cm}

\noindent \textsc{Eduardo V. Teixeira} \hfill \textsc{Jos\'e Miguel Urbano}\\
Universidade Federal do Cear\'a \hfill  CMUC, Department of Mathematics \\
Campus of Pici - Bloco 914 \hfill University of Coimbra \\
Fortaleza - Cear\'a - Brazil  \hfill  3001--501 Coimbra  \\
60.455-760 \hfill Portugal \\
\texttt{teixeira@mat.ufc.br} \hfill \texttt{jmurb@mat.uc.pt}

\end{document}